\documentclass[titlepage]{amsart}
\usepackage{amssymb}
\usepackage{amsthm}
\usepackage{bbm}
\usepackage{graphicx}

\usepackage[colorlinks=true]{hyperref}
\usepackage[all]{xypic}

\usepackage[final]{epsfig}
\usepackage{psfrag}
\usepackage{epstopdf}

\usepackage{amsaddr}

\newtheorem{thm}{Theorem}[section]
\newtheorem{theorem}[thm]{Theorem}

\newtheorem{lemma}[thm]{Lemma}

\newtheorem{prop}[thm]{Proposition}
\theoremstyle{definition}
\newtheorem{definition}{Definition}[section]

\theoremstyle{observation}

\theoremstyle{definition}

\newcommand{\voided}[1]{}

\newcommand{\defin}[1]{{\it #1}}

\newcommand{\N}{\mathbb{N}}
\newcommand{\Q}{\mathbb{Q}}

{\bf}{\it}
{\bf}{\it}
\newtheorem*{riemann-mapping-theorem}{Riemann Mapping  Theorem}{\bf}{\it}
{\bf}{\it}
{\bf}{\it}
{\bf}{\it}
{\bf}{\it}
{\bf}{\it}
{\bf}{\it}
{\bf}{\it}
{\bf}{\it}
{\bf}{\it}

\newenvironment{pf*}[1]{\proof[#1]}{\endproof}
\usepackage{euscript}

\newcommand{\beq}{\begin{equation}}
\newcommand{\eeq}{\end{equation}}

\newtheorem{defn}{Definition}[section]

\newcommand{\tl}{\tilde}

\newcommand{\eps}{\epsilon}

\numberwithin{equation}{section}

\newcommand{\supp}{\operatorname{Supp}}

\newcommand{\cM}{{\mathcal M}}

\newcommand{\cP}{{\mathcal P}}

\newcommand{\RR}{{\mathbb R}}

\newcommand{\NN}{{\mathbb N}}

\newcommand{\ignore}[1]{{}}

 \title[Ulam meets Turing]{Ulam meets Turing: constructing quadratic maps with non-computable SRB measures}

\author{Cristóbal Rojas}
\address{Instituto de Ingeniería Matemática y Computacional,\\ Pontificia Universidad Católica de Chile.\\ Vicuña Mackenna 4860, San Joaquín, Santiago, Chile.   
\\E-mail:  \texttt{\emph{cristobal.rojas{@}mat.uc.cl}}}

\author{Michael Yampolsky}
\address{Department of Mathematics, University of Toronto,\\ 40 St George Street, Toronto, Ontario, Canada. \\ E-mail:  \texttt{\emph{yampol{@}math.toronto.edu}}}

 \thanks{ C.R was partially supported by Grants ANID/FONDECYT  Regular 1230469 and ANID/Basal National Center for Artificial Intelligence CENIA FB210017.  M.Y. was partially supported by NSERC Discovery grant.}
\begin{document}

\begin{abstract}
In 1946, S. Ulam invented Monte Carlo method, which has since become the standard numerical technique for making statistical predictions for long-term behaviour of dynamical systems. We show that this, or in fact any other numerical approach can fail for the simplest non-linear discrete dynamical systems given by the logistic maps $f_{a}(x)=ax(1-x)$ of the unit interval. We show that there exist computable real parameters $a\in (0,4)$ for which almost every orbit of $f_a$ has the same asymptotical statistical distribution in $[0,1]$, but this limiting distribution is not Turing computable. 
\end{abstract}

\maketitle


\section{Introduction}

For all practical purposes, the world around us is not a deterministic one.  Even if a simple physical system can be described deterministically,  say by the laws of Newtonian mechanics, the differential equations expressing these laws typically cannot be solved explicitly. This means that predicting the exact evolution of the system for a long period of time is impossible. A classical example is the famous 3-body Problem, which asks to describe the evolution of a system in which three celestial bodies (the ``Earth'', the ``Sun'', and the ``Moon'') interact with each other via the Newton's force of gravity. Computers are generally not of much help either: of course, a system of ODEs can be solved numerically, but the solution will inevitably come with an error due to round-offs. Commonly, solutions of dynamical systems are very sensitive to such small errors (the phenomenon known as ``Chaos''), so the same computation can give wildly different numerical results.

An extreme example of the above difficulties is the art of weather prediction. A realistic weather model will have such a large number of inputs and parameters that simply running a numerical computation will require a massive amount of computing resources; it is, of course, extremely sensitive to errors of computation. A classical case in point is the Lorenz system suggested by meteorologist Edward Lorenz in 1963 \cite{Lorenz}. It has only three variables and is barely non-linear (just enough not to have an explicit solution), and nevertheless it possesses a chaotic attractor \cite{Tucker} -- one of the first such examples in history of mathematics-- so deterministic weather predictions even in such a simplistic model are practically impossible.

Of course, this difficulty is well known to practitioners, and yet weather predictions are somehow made, and sometimes are even accurate. They are made in the language of statistics (e.g. there is a 40\% chance of rain tomorrow), and are based on what is broadly known as {\it Monte Carlo} technique, pioneered by Ulam and von Neumann in 1946 \cite{URvN,MetUl,Metropolis}. Informally  speaking, we can throw random darts to select a large number of initial values; run our simulation for the desired duration for each of them; then statistically average the outcomes. We then expect these averages to reflect the true statistics of our system. To set the stage more formally, let us assume that we have a discrete-time dynamical system
$$f:D\to D,\text{ where }D\text{ is a finite domain in }\mathbb R^n$$
that we would like to study. Let $\bar x_1,\ldots,\bar x_k$ be $k$ points in $D$ randomly chosen, for some $k>>1$ and consider the probability measure
\begin{equation}\label{eq:m-csum}
  \mu_{k,n}=\frac{1}{kn}\sum^k_{l=1}\sum_{m=1}^n \delta_{f^{\circ m}(\bar x_l)},
\end{equation}
where $\delta_{\bar x}$ is the delta-mass at the point $\bar x\in\mathbb R^n$. The mapping $f$ can either be given by mathematical formulas, or stand for a computer program we wrote to simulate our dynamical system. The standard postulate is then  that for $k,n\to\infty$ the probabilities $\mu_{k,n}$ converge to a limiting statistical distribution that we can use to make meaningful long-term {\it statistical} predictions of our system.

Let us say that a measure $\mu$ on $D$ is a {\it physical measure} of $f$ if its basin $B(\mu)\subset D$ --that is, the set of initial values $\bar x$ for which the weak limit of $\frac{1}{n}\sum_{m=1}^n \delta_{f^{\circ m}(\bar x)}$ equals $\mu$-- has positive Lebesgue measure. This means that the limiting statistics of such points will appear in the averages (\ref{eq:m-csum}) with a non-zero probability. If there is a unique physical measure in our dynamical system, then one random dart in (\ref{eq:m-csum}) will suffice. Of course, there are systems with many physical measures. For instance, Newhouse \cite{Newhouse} showed that a polynomial map $f$ in dimension $2$ can have infinitely many attracting basins, on each of which the dynamics will converge to a different stable periodic regime. This in itself, however, is not necessarily an obstacle to the Monte-Carlo method, and indeed, the empirical belief is that it still succeeds in these cases. 

Our results are most surprising in view of the above computational statistical paradigm. Namely we consider the simplest examples of non-linear dynamical systems: quadratic maps of the interval $[-1,1]$ of the form
$$f_a(x)=ax(1-x),\; a\in(0,4]$$
and find computable values of $a$ for which:
\begin{enumerate}
\item\label{existence} there exists a {\it unique} physical measure $\mu$ and its basin $B(\mu)\subset [0,1]$ has full Lebesgue measure. 
  \item\label{non_computability} the measure $\mu$ is not computable. 

\end{enumerate}

In a version of this work presented at \cite{RY-STOC}, we gave a non-constructive proof of this non-computability phenomenon, but the question of whether such parameters can be produced algorithmically was left open.  The fact that such a parameter $a$ can be computed means that there is a computer program that can approximate its value at any desired precision.  However, there exists no computer program to approximate the probability that an orbit of the corresponding system be in a given rational subinterval of $[0,1]$. Thus, the Monte-Carlo computational approach fails spectacularly for truly simple and explicitly computable maps -- not because there are no physical measures, or too many of them, but because the ``nice'' unique limiting statistics cannot be computed, and thus the averages (\ref{eq:m-csum}) will not converge to anything meaningful in practice.

It is worth drawing a parallel with our recent paper \cite{YRAdvances19}, in which we studied the computational complexity of topological attractors of maps $f_a$. Such attractors  capture the limiting {\it deterministic} behavior of the orbits. They are always computable, and we found that for almost every parameter $a$, the time complexity of computing its attractor is polynomial, although there exist attractors with an arbitrarily high computational complexity. In dynamics, both in theory and in practice, it is generally assumed that long-term statistical properties are simpler to analyze than their deterministic counterparts. From the point of view of computational complexity, this appears to be false.

We note that computability of invariant measures has been studied before \cite{Cristobal, GalHoyRoj09, GaHoRo, BBRY}. In \cite{GalHoyRoj09} for instance the authors construct continuous maps of the circle for which computable invariant measures do not exists. In the context of symbolic systems, there have been some recent works studying the computational properties of the limiting statistics, see e.g. \cite{Sablik}, and of thermodynamic invariants (see e.g. in \cite{HM, BSW}). The computational complexity of individual trajectories in Hamiltonian dynamics has been addressed in e.g. \cite{kawa}. Long-term unpredictability is generally associated with dynamical systems containing embedded Turing machines (see e.g. the works \cite{moore, MK,koiran, BGR, BRS}). Dynamical properties of Turing machines viewed as dynamical systems have similarly been considered (cf. \cite{Kurka, jeandel}). Yet we are not aware of any studies of the limiting statistics in this latter context.  We also point out that a different notion of statistical intractability in dynamics, based on the complexity of a mathematical description of the set of limiting measures, has been introduced and studied in \cite{berger1,berger2}. 

From a practical point of view, some immediate questions arise. Our examples are rare in the one-parameter quadratic family $f_a(x)=ax(1-x)$. However, there are reasons to expect that in more complex multi-parametric, multi-dimensional families, they can become common. Can they be {\it generic} in a natural setting? As the results of \cite{berger2} suggest, the answer may already be ``yes" for quadratic polynomial maps in dimension two.
Furthermore, even in the one-dimensional quadratic family $f_a$ it is natural to ask what the typical computational complexity of the limiting statistics is -- even if it is computable in theory, it may not be in practice.  

\subsection*{Acknowledgement} The authors would like to thank Pierre Berger for helpful discussions.


\section{Preliminaries}

 \subsection*{Statistical simulations and computability of probability measures}
 
We give a very brief summary of relevant notions of Computability Theory and Computable Analysis. For a more in-depth
introduction, the reader is referred to e.g. \cite{BY-book}.
As is standard in Computer Science, we formalize the notion of
an algorithm as a {\it Turing Machine} \cite{Tur}.  
We will call a  function $f:\NN\to\NN$  \emph{computable} (or {\it recursive}), if there exists a Turing Machine  $\cM$ which, upon input $n$, outputs $f(n)$.
Extending algorithmic notions to functions of real numbers was pioneered by Banach and Mazur \cite{BM,Maz}, and
is now known under the name of {\it Computable Analysis}. 
Let us begin by giving the modern definition of the notion of computable real
number,  which goes back to the seminal paper of Turing \cite{Tur}. By identifying $\Q$ with $\N$ through some effective enumeration, we can assume algorithms can operate on $\Q$. Then a real number $x\in\RR$ is called \defin{computable} if there is an algorithm  $M$ which, upon input $n$, halts and outputs a rational number $q_n$ such that  $|q_n-x|<2^{-n}$.
Algebraic numbers or  the familiar constants such as $\pi$, $e$, or the Feigenbaum constant  are computable real numbers. However, the set of all computable real numbers $\RR_C$ is necessarily countable, as there are only countably many Turing Machines. 

We now define computability of functions over $[0,1]$.  Recall that for a continuous function $f\in C_{0}([0,1])$, a modulus of continuity consists of a function $\delta: \Q\cap(0,a)\to\Q\cap(0,a)$ such that $|f(x)-f(y)|\leq \epsilon$ whenever $|x-y|\leq \delta(\epsilon)$. A function $f:[0,1]\to[0,1]$ is \defin{computable} if it has a computable modulus of continuity and there is an algorithm which, provided with a rational number which is $\delta(\epsilon)$-close to $x$, outputs a rational number which is $\epsilon$-close to $f(x)$.  

Computability of probability measures, say over $[0,1]$ for instance, is defined by requiring the ability to compute the expected value of computable functions. 

\begin{defn}\label{1}Let $(f_{i})$ be any sequence of uniformly computable functions over $[0,1]$. A probability measure $\mu$ over $[0,1]$ is \defin{computable} if there exist a Turing Machine $M$ which on input $(i,\epsilon)$ (with $\epsilon \in \Q$) outputs a rational $M(i,\epsilon)$ satisfying
$$
|M(i,\epsilon) - \int f_{i}\,d\mu | < \epsilon.
$$
\end{defn}

We note that this definition it compatible with the notion of weak convergence (see Section \ref{stage}) of measures in the sense that a measure is computable if and only if it can be algorithmically approximated (in the weak topology) to an arbitrary degree of accuracy by measures supported on finitely many rational points and with rational weights.  
Moreover, this definition also models well the intuitive notion of statistical sampling in the sense that a measure $\mu$ is computable if and only if there is an algorithm to convert sequences sampled from the uniform distribution into sequences sampled with respect to $\mu$. 

In this paper, we will be interested in the computability properties of invariant measures of  quadratic maps of the form $ax(1-x)$, with $a\in \RR$. 


\subsection*{Invariant measures of quadratic polynomials and the statement of the main result.}
As before, we denote 
$$f_a(x)=ax(1-x).$$
For $a\in[0,4]$, this quadratic polynomial maps the interval $[0,1]$ to itself. We will view $f_a:[0,1]\to[0,1]$ as a discrete dynamical system, and will denote $f_a^n$ the $n$-th iterate of $f_a$. 

A measure $\mu$ is called \defin{physical} or {\it Sinai-Ruelle-Bowen (SRB)} if 
\begin{equation}\label{srb}
\frac{1}{n}\sum_{k=0}^{n-1}\delta_{f^{k}x} \to \mu 
\end{equation}
for a set of positive Lebesgue measure.  It is known that if a physical measure exists for a quadratic map $f_{a},\,\, a\in [0,4]$, then it is unique and (\ref{srb}) is satisfied for Lebesgue almost all $x\in [0,1]$.

\bigskip
\noindent\textbf{Main Theorem.} \emph{
There exist computable parameters $a\in(0,4)$ for which the quadratic map $f_a(x)=ax(1-x)$ has a physical measure $\mu$ which is not computable. 
	}
\bigskip

\noindent
We note that in the version of this work which was presented at STOC-2020 \cite{RY-STOC}, we showed the existence of an uncountable set of parameters for which the physical measure exists but is not computable, yet the proof was not constructive.


\section{Proof of the Main Theorem}
The proof is based on a delicate construction in one-dimensional dynamics described in \cite{HofKel}, which will allow us to construct maps $f_a$ with physical measures which selectively charge points in a countable set of periodic orbits. To give precise formulations, we will need to introduce some further concepts.

\subsection{Setting the Stage}\label{stage}
It will be convenient to recall that weak convergence of measures on $[0,1]$ is compatible with the notion of {\it Wasserstein-Kantorovich distance},
defined by:

\begin{equation*}
W_{1}(\mu,\nu)=\underset{f\in 1\text{-Lip}([0,1])}{\sup}\left|\int f d\mu-\int f d\nu\right|
\end{equation*}
\noindent where $1\mbox{-Lip}([0,1])$ is the space of Lipschitz functions on $[0,1]$, having Lipschitz constant less than one.

For $a\in[0,4]$ and $x\in[0,1]$, we set
$$\nu_a^n(x)=\frac{1}{n}\sum_{k=0}^{n-1}\delta_{f_a^kx}.$$

We will make use of the following folklore fact (see e.g. \cite{dMvS}):
\begin{prop}
  \label{sink}
  Suppose, for $a\in[0,4]$ the map $f_a$ has a non-repelling periodic orbit of period $p$:
  $$x_0\overset{f_a}{\mapsto}x_1\overset{f_a}{\mapsto}\cdots\overset{f_a}{\mapsto}x_{p-1}\overset{f_a}{\mapsto}x_0,\;\left| \frac{d}{dx}f_a^p(x_0)\right|\leq 1.$$
  Let
  $$\mu\equiv \frac{1}{p}\sum_{k=0}^{p-1}\delta_{x_k}.$$
  Then $\mu$ is the unique physical measure of $f_a$ (so, in particular, the non-repelling orbit is unique); and
  $$W_1(\nu_a^n(x), \mu)\to 0$$
  uniformly on a set of full Lebesgue measure in $[0,1]$. 
  \end{prop}
  
We now introduce a certain collection of periodic orbits. For $a\in(0,4]$ consider the third iterate $g_a\equiv f_a^3$. 
We start by noting that there exists a parameter value $c\in (3.85,4)$ such that the following holds: $$g_c(0.5)\neq g_c^2(0.5)=g_c^3(0.5).$$
If we denote $\beta_c=g_c^2(0.5)$, then $\beta'_c\equiv g_c(0.5)=1-\beta_c$, and denoting $I_c\equiv [\beta'_c,\beta_c]\ni 0.5$, we have
$$g_c(I_c)=I_c;$$
both endpoints of $I_c$ map to $\beta_c$. The restriction $g_c|_{I_c}$ maps both halves $L_c=[\beta',0.5]$ and $R_c=[0.5,\beta]$  of the  interval $I_c$ onto the whole $I_c$ in a monotone fashion (that is, it folds $I_c$ over itself).

\begin{figure}[ht]
  \includegraphics[width=0.6\textwidth]{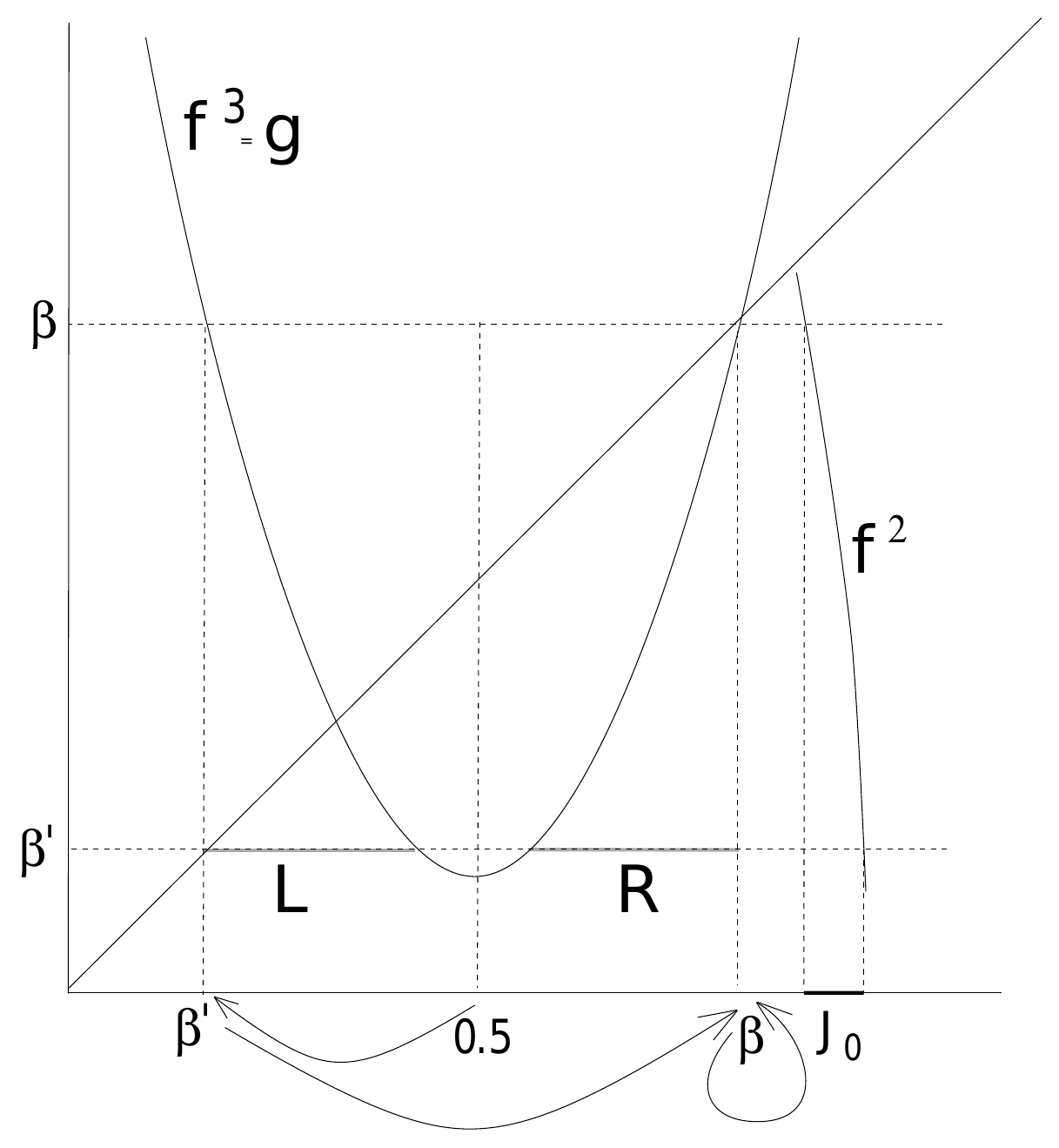}
\caption{\label{fig:iterates}Some iterates of $f\equiv f_a$ for $a\in(c,4]$ restricted to the appropriate intervals (we drop the subscript $a$ for simplicity in all notations in the figure).}
 \end{figure} 

For $a\in [c,4]$, there exists a continuous branch $\beta_a$ of the fixed point
$$g_a(\beta_a)=\beta_a,$$
and we again set $\beta'_a=1-\beta_a$  (so $g_a(\beta'_a)=\beta_a$), and $I_a\equiv[\beta'_a,\beta_a]$. Now, if $a\in(c,4]$, the image
  $$g_a(I_a)\supsetneq I_c,\text{ with }g_a(0,5)<\beta'_a.$$
  Thus, there is a pair of sub-intervals $L_a=[\beta'_a,l_a]$, $R_a=[r_a,\beta_a]$ inside $I_a$ which are mapped monotonely over $I_a$ by $g_a$ (the endpoints $l_a$, $r_a$ are both mapped to $\beta'_a$. See Figure~\ref{fig:iterates} for an illustration.

  Assigning values $0$ to $L_a$, and $1$ to $R_a$ we obtain symbolic dynamics on the set of points
  $$C_a\equiv \{x\in I_a\text{ such that }g_a^n(x)\in I_a\text{ for all }n\in\NN\}.$$
  If $a=c$, then, of course, $C_a=I_a$. For larger values of $a$,  we obtain a Cantor repeller (see e.g. \cite{dMvS}):
  \begin{prop}\label{prop-cantor}
If $a\in(c,4)$ then $C_a$ is a Cantor set, and the symbolic dynamics conjugates $g_a|_{C_a}$ to the full shift $\Omega=\{0,1\}^\N$.
    \end{prop}
  In particular, every periodic sequence of $0$'s and $1$'s corresponds to a unique periodic orbit in $C_a$ with this symbolic dynamics. These orbits clearly move continuously with $a$, and can be easily computed given $a$ and the symbolic sequence as the unique fixed points of the corresponding monotone branches of iterates of $g_a$. We enumerate all periodic sequences of $0$'s and $1$'s in $\Omega$ as follows. A sequence with a smaller period will precede a sequence with a larger period. Within the sequences of the same period, the ordering will be lexicographic, based on  
  the convention $1\prec 0$. 
We let $$\{p^1_{a,n},\ldots,p^{k_n}_{a,n}\}$$ be the periodic orbit of $g_a$ in $C_a$ which corresponds to the $n$-th symbolic sequence in this ordering (note that the first one is $\beta_a$). We denote
  $$\text{Per}_a(n)=\cup_{j=0}^2 f^j_a(\{p^1_{a,n},\ldots,p^{k_n}_{a,n}\}),$$
  which is, clearly, a periodic orbit of $f_a$.

\subsection{Constructing Non Computable Physical Measures}

Our arguments will be based on the construction of S.~Johnson \cite{Johnson}, which was further developed by  F.~Hofbauer and G.~Keller in \cite{HofKel}.

\subsubsection{Johnson's construction}
To give an idea of Johnson's construction in our setting, we start with the following  example:

    \medskip\noindent{\bf Example 1}: {\it The set $$\Omega_a(0.5)=\lambda_a(1)=\frac{1}{3}(\delta_{\beta_a}+\delta_{f_a(\beta_a)}+\delta_{f^2_a(\beta_a)})$$ and
      $\Omega(x)=\Omega(0.5)$ for almost every $x$ (compare with Theorem~1 of \cite{HofKel}).}

    \medskip\noindent
    Consider again Figure~\ref{fig:iterates} as an illustration. We note that there exists an interval $J_0$ to the right of the fixed point $\beta_a$ such that the following holds:
    \begin{itemize}
    \item $f^2_a(J_0)\Supset [\beta_a',\beta_a]$;
      \item Denote by $\psi_a$ the branch of $g_a^{-1}$ which fixes $\beta_a$. Then the interval $J_0$ is contained in the domain of definition of $\psi_a$. Thus, there is an orbit $$J_{-n}\equiv \psi_a^n(J_0) \to\beta_a \;(\text{here }f_a^{3n}(J_{-n})=J_0).$$
      \end{itemize}
    Moving the parameter $a\in(c,4]$, we can place the image $g_a^2(0.5)$ at any point of  $J_{-{n_1}}$, for an arbitrary $n_1$. If the value of $n_1$ is large, then the $g_a$-orbit of $0.5$ will spend a long time in a small neighborhood of  $\beta_a$, before hitting some $x_1\in J_0$.
  Adjusting the position of
      $a\in(c,4]$, we can ensure that $f_a^2(x_1)$ is inside $J_{-n_2}$ for an even larger $n_2$, so the orbit returns to an even smaller neighborhood of $\beta_a$ where it will spend an even longer time. Continuing increasing $n_k$'s as needed so the orbit of $0.5$ spends most of its time in ever smaller neighborhoods of $\beta_a$, we can ensure that the averages
    $\nu_a^n(0.5)=\frac{1}{n}\sum_{k=0}^{n-1}\delta_{f_a^k(0.5)} $ converge to the delta masses supported on the orbit of $\beta_a$.

   Proceeding in this way,
    for an arbitrarily large $l\in\NN$ and $x\in [\beta_a',\beta_a]$ we can find $a\in(c,4]$ and $m>2^{-l}$ such that:
    \begin{enumerate}
    \item the distance
      $$W_1(\nu_a^{m}(0.5)-\lambda_a(1))<2^{-l};$$
    \item the iterate $f_a^m(0.5)$ lies in $J_0$;
    \item the next iterate $f_a^{m+1}(0.5)=0.5$.  

    \end{enumerate}

    Property (3) ensures that the critical point $0.5$ is periodic with period $m+1$. Since $f_a'(0.5)=0$, we have
		$(f_a^{m+1})'(0.5)=0$, so this is a (super)attracting periodic point. Proposition~\ref{sink} implies that the physical measure $\mu_a$ for $f_a$ is supported on the orbit of $0.5$, and thus
    $$W_1(\mu_a-\lambda_a(1))<2\cdot 2^{-l}.$$

    In what follows, we will find it more convenient to replace property (3) with the property (3'):
    \begin{itemize}
    \item[(3')] the next iterate $f_a^{m+1}$ has a parabolic fixed point $p$ with $\frac{d}{dx}f_a^{m+1}(p)=1$ and $f_a^{(m+1)j}\underset{j\to\infty}{\longrightarrow}(0.5)$.  
    \end{itemize}
    This is easily achieved by replacing the center of a hyperbolic component obtained in (3) with a boundary point of the same hyperbolic component.
    
    Again, by Proposition~\ref{sink} and considerations of continuity, there exist $n>m$ and $\eps>0$ such that for any $a'$ with $|a'-a|<\eps$, we have
    $$W_1(\nu_{a'}^{n}(x)-\lambda_a(1))<4\cdot 2^{-l}$$
    for any $x$ in a set of length $1-2^{-l}$.

    Assuming $\eps$ is small enough, we again have $g_{a'}(0.5)$ slightly to the right of $\beta_{a'}$ and we can repeat the above steps inductively to complete the construction. 
    
        As a next step, we  construct an asymptotic  measure supported on two periodic orbits:

        \medskip
        \noindent
            {\bf Example 2:} {\it the set $\Omega(0.5)=a_1\lambda_1(1)+a_2\lambda_a(n)$ for $n>1$ and $a_1+a_2=1$.}

            \medskip\noindent
            Let $p\in \text{Per}_a(n)$ and, as before, denote by $3k_n$ its period. Letting $\phi$ denote the branch of $f_a^{-3k_n}$ fixing $p$,
            we again find a $\phi$-orbit
            $$J'_0\mapsto J'_{-1}\mapsto J'_{-2}\mapsto\cdots,\text{ with }J'_{-k}\to p,$$
            where
            $$f_a^s(J_0)\Supset [\beta'_a,\beta_a]$$
            for a univalent branch of the iterate $f_a^s$.

            Now we can play the same game as in Example 1, alternating between entering the orbit $J_{-k}$ close to the point $\beta_a$, and the orbit $J'_{-k}$ close to $p$. In this way, we can achieve the desired limiting asymptotics with any values $a_1, a_2$.

            Let us distill the inductive step in the above constructions as follows:
            \begin{definition}\label{j-parameter}
Let $a\in[c,4]$.
We will say that $a$ is an admissible parameter if the following properties hold:
\begin{enumerate}
\item $a$ is parabolic, that is, there exists $p\in[0,1]$ such that $f_a^j(p)=p$ and $Df^j_a(p)=1$;
  \item There exists an interval $J$ which is compactly contained in the immediate domain of repulsion of $p$ such that an iterate $f_a^i$ diffeomorphically maps $J$ onto $[\beta_a',\beta_a]$.

  \end{enumerate}

            \end{definition}

The following is a computable version of the construction from \cite{Johnson} (see also \cite{HofKel}):
\begin{lemma}\label{key_lemma} Let $a\in(c,4]$ be an admissible parameter, $\{l_n\}_n$ be a uniformly computable sequence of positive real numbers that add up to 1,  $\delta>0$ be a rational number, and $k>0$ be any positive integer. Then we can compute, uniformly in the above data, an admissible parameter $\bar{a}\in(a-\delta,a+\delta)$, a rational $\epsilon>0$, and an integer $m$ such that for all $n\in \NN$ and $r\in [\bar{a}-\epsilon, \bar{a}+\epsilon]$ we have:
  \begin{equation}\label{lem-1}
|\nu_r^m(x)(\textup{Per}_r(n)) - l_n| < 2^{-k} 
\end{equation}
for any $x$ in a set of measure at least $1-2^{-k}$. 
\end{lemma}

\begin{proof}
  Let $N\in\NN$ have the property
  $$\underset{1\leq n\leq N}{\sum}l_n>1-2^{-(k+2)}.$$
  From \cite{Johnson,HofKel}, there exists an admissible parameter $\bar{a}\in(a-\delta,a+\delta)$ such that
  \begin{equation}\label{lem-2}
|\mu_{\bar a}(\textup{Per}_r(n)) - l_n| < 2^{-(k+2)} \quad \text{ for } 1\leq n\leq N. 
  \end{equation}

  Note that such a parameter can be found by a greedy search. Indeed, parabolic parameters are specified by a set of algebraic conditions, and so are the endpoints of an interval $J$ in part (2) of Definition~\ref{j-parameter}. Thus a countable search eventually finds each admissible parameter with any desired precision.
  Similarly, if a parameter satisfies the inequality (\ref{lem-2}), then this can be eventually verified by performing the calculation of the left hand side with precision $2^{-j}$ for some sufficiently large $j$. The countable search which at step $j$ attempts to verify (\ref{lem-2}) with precision $2^{-j}<2^{-(k+2)}$ for admissible parabolic parameters with periods not greater than $j$ will eventually succeed in finding such $\bar a$.

  Once $\bar a$ is found, it is a standard exercise on estimating the condition number (see, for example, \cite{BSS-cond})
   to find $m$ and $\epsilon$.
  \end{proof}

\subsubsection{Proof of the main result}

\begin{theorem}There exists computable parameters $a\in(2,4)$ for which the map $f_a$ has a unique physical measure $\mu_a$ whose computability is equivalent to solving the Halting Problem. 
\end{theorem}
\begin{proof}

Let $M_n$ be an effective enumeration of all Turing Machines and let $\mu$ be the probability measure defined by:  

\[
\mu(\text{Per}(m))=\begin{cases}  2^{-n} &  \text{ if } m = 2n \text{ and } M_n \text{ halts};  \\ 2^{-n} & \text{ if} \quad m = 2n-1 \text{ and } M_n \text{ doesn't halt }\\
0 & \text{ otherwise. } \end{cases}
\]
We note that for all $n \in \N$ it holds \[\mu\big(\text{Per}(2n) \cup \text{Per}(2n-1)\big) = 2^{-n} \] so that $\mu$ is a probability measure supported on the collection $\{\text{Per}(n)\}_{n\in\N}$. Moreover, the measure $\mu$ cannot be computable. This follows from the fact that for any computable probability measure, the probability of a finite collection of computable points can be uniformly computed from above. Indeed, if $C=\{x_i\}_{i=1}^s$ is such a collection, then one can consider the collection of computable functions $\varphi_l(x)$ which equal 1 at every point  of $C$, 0 outside a $2^{-l}$-neighbourhood of $C$, and that are linear in between. For a computable probability measure $\nu$, by computing the values $\int \varphi_l(x)\, d\nu = q(l)$, we obtain a computable sequence $q(l)$ which decreases to $\nu(C)$ as $l\to\infty$. 
Suppose now that $\mu$ is a computable measure. Then, for each machine $M_n$, we can uniformly compute a sequence $q_1(l)$ which decreases to $\mu(\text{Per}(2n))$, and  a sequence $q_2(l)$ which decreases to $\mu(\text{Per}(2n-1))$. Therefore, since one of these sequences decreases to 0, by comparing them against $2^{-n}$, we can decide whether $\mu(\text{Per}(2n))<2^{-n}$ which contradicts the undecidability of the Halting Problem.  

We now proceed with the construction and show how to compute a parameter $a$ such that $\mu_a=\mu$.  The parameter $a$ will be the intersection of a nested sequence of closed intervals inductively constructed. To start, we choose any computable parabolic parameter $c'>c$ and consider the sequence of weights $\{l^0_n\}_n$ given by:
\[
l^0_{m}=\begin{cases}2^{-n} &\text{ if } m=2n-1 \\ 0 &\text{ otherwise. }\end{cases}
\] 

Note that this sequence of weights corresponds to the measure $\mu$ in the case where none of the machines ever halt.  By Lemma \ref{key_lemma} we can compute a parabolic parameter $a_0\in (c,c')$, a positive rational $\epsilon_0<a_0-c$ and an integer $m_0$ such that for all $n\in \N$ and all $r\in I_0=[a_0-\epsilon_0,a_0]$ it holds
\[
|\nu_r^{m_0}(x)(\text{Per}_r(n)) - l^0_n|<2^{-1}
\]
for any $x$ in a set of measure at least $1-2^{-1}$. Let $M_{n_1}$ be the first machine that halts. By Lemma  \ref{key_lemma} again, we can compute a parabolic point $a_1\in I_0$, a rational $\epsilon_1<a_1-(a_0-\epsilon_0)$ and an integer $m_1$ such that for all $r\in I_1=[a_1-\epsilon_1]\subset I_0$ it holds:
\[
|\nu_r^{m_1}(x)(\text{Per}_r(n)) - l^1_n|<2^{-2}
\]
for any $x$ in a set of measure at least $1-2^{-2}$, where $\{l^1_n\}_{n\in \N}$ is defined by
\[
l^1_n=\begin{cases}  l^0_n &   \text{ if } n \notin \{2n_{1}, 2n_{1}-1\}  \\ 2^{-n_1}  & \text{ if }  n = 2n_1\\ 0& \text{ if } n=2n_1-1 \end{cases}
\]
and corresponds to ``switching" the weight $2^{-n_1}$ from Per$(2n_1-1)$ to Per$(2n_1)$. We proceed inductively, switching weights each time a new machine halts. More precisely, assuming the interval  $I_k$ from stage $k$ has already been constructed, and denoting by $M_{n_{k+1}}$ the $(k+1)-$th machine that halts, we can use Lemma \ref{key_lemma} to compute a new interval $I_{k+1}$ and an integer $m_{k+1}$ satisfying the following properties:
\begin{itemize}
\item[i)]$I_{k+1}$ is strictly contained in $I_k$ and of at most half the size;   
\item[ii)] the right end of $I_{k+1}$ is a parabolic point $a_{k+1}$ and
\item[iii)] for all $n\in\N$ and $r\in I_{k+1}$ it holds
\[
|\nu_r^{m_{k+1}}(x)(\text{Per}_r(n)) - l^{k+1}_n|<2^{-(k+2)}
\]
for any $x$ in a set of measure at least $1-2^{-(k+2)}$, where $\{l^{k+1}_n\}_{n\in \N}$ is defined by
\[
l^{k+1}_n=\begin{cases}  l^k_n &  \text{ if } n \notin \{2n_{k+1}, 2n_{k+1}-1\}  \\ 2^{-n_{k+1}} & \text{ if } n = 2n_{k+1} \\ 0 & \text{ if } n=2n_{k+1}-1. \end{cases}
\]
\end{itemize}
We now let 
\[
a=a_\infty = \bigcap_k I_k.
\]

Since $I_k$ shrinks exponentially, we have that $a$ is a computable parameter. We now show that $\mu_{a}=\mu$ as announced. For any $n\in\N$, let $M_{n_l}$ be the last machine that halts among the first $n$ machines. That is, a machine among the first $n$ that halts, must halt before $M_{n_l}$. Note that in this case, for every $i\leq n$ we have
\[l^k_i = \mu\big(\text{Per}(i)\big)\]
for all $k>l$. Since $a$ belongs to all the $I_k$, we have that for every $\epsilon>0$ \[|\nu_a^{m}(x)(\text{Per}_a(i)) - \mu\big(\text{Per}_a(i)\big)|<\epsilon\] holds for all $m>m_l$, all $i\leq n$ and any $x$ in a set of measure at least $1-\epsilon$. It follows that for all $i\leq n$
\[
\nu_a^\infty(x)(\text{Per}_a(i))=\mu\big(\text{Per}_a(i)\big) \quad \text{ for almost every }x. 
\]
Since this holds for all $n$, we obtain that $\nu_a^{\infty}(x)=\mu$ for almost every $x$, and thus $\mu$ is the unique physical measure of $f_a$ as it was to be shown. 

\end{proof}

\section{Conclusion}
\label{sec:conclusion}

Ever since the first numerical studies of chaotic dynamics appeared in the early 1960's (such as the work of Lorenz \cite{Lorenz}), it has become commonly accepted among practitioners that computers cannot, in general, be used to make deterministic predictions about future behavior of nonlinear dynamical systems. Instead, the standard practice now is to make statistical predictions. This approach is based on the Monte Carlo method, pioneered by Ulam and von Neumann at the dawn of the computing age. It is universal and powerful -- and only requires access to the dynamical system as a black box, which is then subjected to a number of statistical trials.
Applications of the Monte Carlo technique are ubiquitous, ranging from weather forecasts to simulating nuclear weapons tests (nuclear weapons design was, of course, the original motivation of its inventors).

Our result raises a disturbing possibility that even for the most simple family of examples of non-linear dynamical systems the Monte Carlo method can fail. Given one of our examples as a black box, no algorithm can find its limiting statistics. How common such examples are in higher-dimensional families of dynamical systems, and whether one is likely to encounter one {\it in practice} remain exciting open questions.

\bibliographystyle{alpha}
\bibliography{bib_RY}

\begin{thebibliography}{BBRY11}

\bibitem[BB19]{berger2}
P.~Berger and J.~Bochi.
\newblock On emergence and complexity of ergodic decompositions.
\newblock {\em ar{X}iv preprint math/1901.03300}, 2019.

\bibitem[BBRY11]{BBRY}
I.~Binder, M.~Braverman, C.~Rojas, and M.~Yampolsky.
\newblock Computability of {B}rolin-{L}yubich measure.
\newblock {\em Commun. Math. Phys.}, 308:743--771, 2011.

\bibitem[BCSS98]{BSS-cond}
L.~Blum, F.~Cucker, M.~Shub, and S.~Smale.
\newblock The condition number for nonlinear problems.
\newblock In {\em Complexity and Real Computation}, pages 217--236. Springer New York, New York, NY, 1998.

\bibitem[Ber17]{berger1}
P.~Berger.
\newblock Unpredictability of dynamical systems and non-typicality of the finiteness of the number of attractors in various topologies.
\newblock {\em Tr. Mat. Inst. Steklova}, 297:7--37, 2017.
\newblock English version published in Proc. Steklov Inst. Math. {{\bf{297}}} (2017), no. 1, 1--27.

\bibitem[BGR12]{BGR}
M.~Braverman, A.~Grigo, and C.~Rojas.
\newblock Noise vs computational intractability in dynamics.
\newblock In {\em Proceedings of the 3rd Innovations in Theoretical Computer Science Conference}, ITCS '12, pages 128--141, New York, NY, USA, 2012. ACM.

\bibitem[BM37]{BM}
S.~Banach and S.~Mazur.
\newblock Sur les fonctions caluclables.
\newblock {\em Ann. Polon. Math.}, 16, 1937.

\bibitem[BRS15]{BRS}
M.~Braverman, C.~Rojas, and J.~Schneider.
\newblock Space-bounded {C}hurch-{T}uring thesis and computational tractability of closed systems.
\newblock {\em Physical Review Letters}, 115(9), August 2015.

\bibitem[BW18]{BSW}
M.~Burr, M.;~Schmoll and C.~Wolf.
\newblock On the computability of rotation sets and their entropies.
\newblock {\em Ergodic Theory and Dynamical Systems}, pages 1--35, 2018.

\bibitem[BY08]{BY-book}
M.~Braverman and M.~Yampolsky.
\newblock {\em Computability of {J}ulia sets}, volume~23 of {\em Algorithms and {C}omputation in {M}athematics}.
\newblock Springer, 2008.

\bibitem[dMvS93]{dMvS}
W.~de~{M}elo and S.~van {Strien}.
\newblock {\em One-dimensional dynamics}.
\newblock Springer-{V}erlag, 1993.

\bibitem[GHR10]{GalHoyRoj09}
S.~Galatolo, M.~Hoyrup, and C.~Rojas.
\newblock Dynamics and abstract computability: computing invariant measures.
\newblock {\em Discr. Cont. Dyn. Sys. Ser A}, 2010.

\bibitem[GR11]{GaHoRo}
M.~G{\'a}cs, P.;~Hoyrup and C.~Rojas.
\newblock Randomness on computable probability spaces -- a dynamical point of view.
\newblock {\em Theory Comput. Syst.}, 48(465), 2011.

\bibitem[HdMS16]{Sablik}
Benjamin Hellouin~de Menibus and Mathieu Sablik.
\newblock Characterisation of sets of limit measures after iteration of a cellular automaton on an initial measure.
\newblock {\em Ergodic Theory and Dynamical Systems}, 38(2):601--650, 2016.

\bibitem[HK90]{HofKel}
F.~Hofbauer and G.~Keller.
\newblock Quadratic maps without asymptotic measure.
\newblock {\em Comm. {M}ath. {P}hys.}, 127:319--337, 1990.

\bibitem[HM10]{HM}
M.~Hochman and T.~Meyerovitch.
\newblock Characterization of the entropies of multidimensional shifts of finite type.
\newblock {\em Annals of Mathematics}, 171(3):2011--2038, 2010.

\bibitem[Jea14]{jeandel}
E.~Jeandel.
\newblock Computability of the entropy of one-tape {T}uring machines.
\newblock In {\em 31st {I}nternational {S}ymposium on {T}heoretical {A}spects of {C}omputer {S}cience (STACS)}, volume~25 of {\em LIPIcs. Leibniz Int. Proc. Inform.}, pages 421--432. 2014.

\bibitem[Joh87]{Johnson}
S.~Johnson.
\newblock Singular measures without restrictive intervals.
\newblock {\em Commun. {M}ath. {P}hys.}, 110:185--190, 1987.

\bibitem[KCG94]{koiran}
P.~Koiran, M.~Cosnard, and M.~Garzon.
\newblock Computability with low-dimensional dynamical systems.
\newblock {\em Theoret. Comput. Sci.}, 132(1-2):113--128, 1994.

\bibitem[KTZ18]{kawa}
Akitoshi Kawamura, Holger Thies, and Martin Ziegler.
\newblock Average-case polynomial-time computability of {H}amiltonian dynamics.
\newblock In {\em 43rd International Symposium on Mathematical Foundations of Computer Science, {MFCS} 2018, August 27-31, 2018, Liverpool, {UK}}, pages 30:1--30:17, 2018.

\bibitem[Kur97]{Kurka}
Petr Kurka.
\newblock On topological dynamics of {T}uring machines.
\newblock {\em Theoret. Comput. Sci.}, 174:203--2016, 1997.

\bibitem[Lor63]{Lorenz}
E.~N. Lorenz.
\newblock Deterministic nonperiodic flow.
\newblock {\em J. Atmos. Sci.}, 20:130--141, 1963.

\bibitem[Maz63]{Maz}
S.~Mazur.
\newblock {\em Computable {A}nalysis}, volume~33.
\newblock Rosprawy Matematyczne, Warsaw, 1963.

\bibitem[Met87]{Metropolis}
N.~Metropolis.
\newblock The beginning of the {M}onte {C}arlo method.
\newblock {\em Los {A}lamos {S}cience {S}pecial {I}ssue}, pages 125--130, 1987.

\bibitem[MK99]{MK}
C.~Moore and P.~Koiran.
\newblock Closed-form analytic maps in one and two dimensions can simulate universal {T}uring machines.
\newblock {\em Theoret. Comput. Sci.}, 210(1):2217--223, 1999.

\bibitem[Moo91]{moore}
C.~Moore.
\newblock Generalized shifts: unpredictability and undecidability in dynamical systems.
\newblock {\em Nonlinearity}, 4(2):199--230, 1991.

\bibitem[MS49]{MetUl}
N.~Metropolis and Ulam. S.
\newblock The {M}onte {C}arlo method.
\newblock {\em Journal of the {A}merican {S}tatistical {A}ssociation}, 44:335--341, 1949.

\bibitem[New74]{Newhouse}
S.~Newhouse.
\newblock Diffeomorphisms with infinitely many sinks.
\newblock {\em Topology}, 13:9--18, 1974.

\bibitem[Roj08]{Cristobal}
C.~Rojas.
\newblock {\em Randomness and ergodic theory: an algorithmic point of view}.
\newblock PhD thesis, Ecole Polytechnique, 2008.

\bibitem[RY19]{YRAdvances19}
Cristobal Rojas and Michael Yampolsky.
\newblock Computational intractability of attractors in the real quadratic family.
\newblock {\em Advances in Mathematics}, 349:941 -- 958, 2019.

\bibitem[RY20]{RY-STOC}
C.~Rojas and M.~Yampolsky.
\newblock How to lose at {M}onte {C}arlo: a simple dynamical system whose typical statistical behavior is non-computable.
\newblock {\em Proceedings of the 52nd Annual ACM SIGACT Symposium on Theory of Computing}, 2020.

\bibitem[Tuc02]{Tucker}
W.~Tucker.
\newblock A rigorous {O}{D}{E} solver and {S}male's 14th problem.
\newblock {\em Found. Comp. Math.}, 2:53--117, 2002.

\bibitem[Tur36]{Tur}
A.~M. Turing.
\newblock On computable numbers, with an application to the {E}ntscheidungsproblem.
\newblock {\em Proceedings, London Mathematical Society}, pages 230--265, 1936.

\bibitem[URvN47]{URvN}
S.~Ulam, R.D. Richtmyer, and J.~von Neumann.
\newblock Statistical methods in neutron diffusion.
\newblock {\em Los {A}lamos {S}cientific {L}aboratory report {L}{A}{M}{S} 551}, 1947.

\end{thebibliography}

\end{document}